\documentclass[a4paper]{amsart}
\usepackage[utf8]{inputenc}

\title[Continuity of the stabilizer map]
  {Continuity of the stabilizer map and irreducible extensions}

\author{Adrien Le Boudec}
\address{CNRS, UMPA--ENS Lyon \\
  46 all\'ee d'Italie \\
  69364 Lyon \\
  France}
\email{adrien.le-boudec@ens-lyon.fr}

\author{Todor Tsankov}
\address{
  Institut Camille Jordan \\
  Universit\'e Claude Bernard Lyon 1 \\
  43, boulevard du 11 novembre 1918 \\
  69622 Villeurbanne \textsc{cedex} \\
  France
  -- and --
  Institut Universitaire de France}
\email{tsankov@math.univ-lyon1.fr}


\usepackage{hyperref}
\usepackage[initials,shortalphabetic]{amsrefs}

\usepackage[T1]{fontenc}
\usepackage[osf]{mathpazo}
\usepackage{eucal} 

\usepackage{my-macros}
\usepackage{enumitem}
\setlist[enumerate,1]{label=(\roman*), font=\normalfont}

\DeclareMathOperator{\RO}{RO}

\DeclareMathOperator{\Sub}{Sub}
\DeclareMathOperator{\Stab}{Stab}

\DeclareMathOperator{\Sam}{Sa}

\begin{document}

\begin{abstract}
  Let $G$ be a locally compact group. For every $G$-flow $X$, one can consider the stabilizer map $x \mapsto G_x$, from $X$ to the space $\mathrm{Sub}(G)$ of closed subgroups of $G$. This map is not continuous in general. We prove that if one passes from $X$ to the universal irreducible extension of $X$, the stabilizer map becomes continuous. This result provides, in particular, a common generalization of a theorem of Frolík (that the set of fixed points of a homeomorphism of an extremally disconnected compact space is open) and a theorem of Veech (that the action of a locally compact group on its greatest ambit is free). It also allows to naturally associate to every $G$-flow $X$ a stabilizer $G$-flow $\mathrm{S}_G(X)$ in the space  $\mathrm{Sub}(G)$, which
  generalizes the notion of stabilizer uniformly recurrent subgroup associated to a minimal $G$-flow introduced by Glasner and Weiss. 
\end{abstract}

\maketitle

\section{Introduction}

Let $G$ be a topological group. Recall that a \df{$G$-flow} is a continuous action $G \actson X$ on a compact space $X$ (all our compact spaces are Hausdorff). A $G$-flow is \df{minimal} if every orbit is dense. A continuous, $G$-equivariant map $\pi \colon Y \to X$ between $G$-flows is called a \df{$G$-map}. If $\pi$ is surjective, we also say that $Y$ is an \df{extension} of $X$, or that $X$ is a \df{factor} of $Y$.

A map $\pi \colon Y \to X$ between compact spaces is called \df{irreducible} if every non-empty open $U \sub Y$ contains the fiber $\pi^{-1}(\set{x})$ for some $x \in X$, or, equivalently, if the image of any proper closed subset of $Y$ is a proper subset of $X$. Irreducible maps were studied by Gleason~\cite{Gleason1958}, who proved that to every compact space $X$, one can associate an extremally disconnected compact space $\hat X$, the Stone space of the Boolean algebra $\RO(X)$ of regular open subsets of $X$, with an irreducible map $\hat X \to X$ which is universal with respect to irreducible maps $Y \to X$. Recall that a space is \df{extremally disconnected} if the closure of every open subset is clopen. 

An extension $\pi \colon Y \to X$ between $G$-flows is called \df{irreducible} if $\pi$ is irreducible as a map between topological spaces. The extension $\pi$ is called \df{highly proximal} if one can compress any fiber of $\pi$ to a point by applying elements of $G$; more precisely, if for every $x \in X$ there exists a net $(g_i)$ of elements of $G$ such that $g_i \cdot \pi^{-1}(\set{x})$ converges to a singleton in the Vietoris topology on the closed subsets of $Y$. These notions were studied by Auslander and Glasner in \cite{Auslander1977} where it was proved that they are equivalent if $X$ and $Y$ are minimal. However, they are different if $Y$ is not minimal (cf. Example~\ref{ex:highly-prox}) and in this paper, we will mostly be interested in the notion of an irreducible extension. Irreducible extensions are thought as being rather small extensions and they preserve many dynamical properties such as minimality, proximality, strong proximality, and disjointness. When the spaces $Y$ and $X$ are metrizable, an extension $\pi \colon Y \to X$ is irreducible iff it is \df{almost one-to-one} (i.e., the set $\set[\big]{y \in Y : \pi^{-1}(\set{\pi(y)}) = \set{y}}$ is dense in $Y$). Almost one-to-one extensions are an important tool in topological dynamics (used, for example to construct symbolic representations of continuous systems), and the notion of an irreducible extension is the appropriate generalization that allows the existence of universal objects and the development of a general theory.

For every $G$-flow $X$, there exists a $G$-flow $\hat X_G$ and an irreducible extension $\pi_X \colon \hat X_G \to X$ with the following universal property: for every irreducible extension $\pi \colon Y \to X$, there exists a $G$-map $p \colon \hat X_G \to Y$ such that $\pi  \circ p = \pi_X$. Moreover, $\hat X_G$ is unique up to isomorphism. It is called the \df{universal irreducible extension} of $X$. For minimal flows, the existence and uniqueness of $\hat X_G$ were established in \cite{Auslander1977} and the general case is due to Zucker~\cite{Zucker2021}. In \cite{Zucker2021}, following the terminology of \cite{Auslander1977} for minimal flows, this extension was called the \df{universal highly proximal extension}; however, in view of the non-equivalence of high proximality and irreducibility for extensions of non-minimal flows and the fact that an irreducible extension is not necessarily proximal (for example, for actions of the trivial group), we prefer to use different names for the two notions.

The universal irreducible extension can be viewed as a type of completion (cf. Section~\ref{sec:univ-irred-extens}), so we call a $G$-flow $X$ \df{Gleason complete} if $\hat X_G = X$. Equivalently, $X$ is Gleason complete if $X$ admits no non-trivial irreducible extensions. The correspondence $X \mapsto \hat X_G$ is idempotent and its image is the class of Gleason complete $G$-flows. Thus the class of $G$-flows is partitioned into equivalence classes, where $X$ and $Y$ are equivalent if they admit a common irreducible extension; or equivalently if $\hat X_G$ and $\hat Y_G$ are isomorphic. Each class contains a unique representative that is Gleason complete.

If $X$ is minimal, being Gleason complete is equivalent to being \df{maximally highly proximal} in the sense of \cite{Auslander1977}. In \cite{Zucker2021}, the term \df{maximally highly proximal}, or \df{MHP} (cf. \cite{Zucker2021}*{Proposition~3.5}) is used even for non-minimal flows with the same meaning as our \emph{Gleason complete}.

For discrete groups, the construction of $\hat X_G$ reduces to the one by Gleason, and we have that $\hat X_G = \hat X$ \cite[Th.\ 3.2]{Gleason1958}. In this setting, a $G$-flow $X$ is Gleason complete iff it is extremally disconnected. This depends only on the topology of $X$, and not on $G$. This is no longer true for non-discrete groups. Examples of Gleason complete flows that arise in the non-discrete setting are $X = G/H$, where $H$ is a closed, cocompact subgroup of $G$, and $G$ acts on $X$ by left translations. Gleason complete flows of Polish groups were extensively studied by Zucker (under the name of MHP flows) in \cite{Zucker2021}, where many more interesting examples can be found. More general topological groups were considered by Basso and Zucker in \cite{Basso2021a}. 

The equivalence relation of having the same universal irreducible extension and the notion of Gleason complete flow are useful to express certain rigidity properties among $G$-flows. An instance of this is a theorem of Rubin that asserts that any two $G$-flows that are faithful and \df{micro-supported} have a common irreducible extension \cite{Rubin-loc-mov-book}. Combined with \cite[Prop.~2.3]{Caprace2023}, this implies that every group $G$ that admits a faithful micro-supported $G$-flow admits exactly one faithful micro-supported $G$-flow that is Gleason complete. For certain non-discrete, totally disconnected locally compact groups, this flow is the Stone space of the centralizer lattice of $G$, a Boolean algebra constructed from the local structure of the group \cite[Th.~II]{CRW-part1}, \cite{Caprace2023}. See the references above for the definition of a ``micro-supported'' action and more details.

\subsection*{The main result}

In certain contexts, Gleason complete flows are better behaved than general flows. The main result of this paper is an illustration of such a situation. For the remainder of the introduction, we suppose that $G$ is a locally compact group, and we denote by $\Sub(G)$ the space of closed subgroups of $G$. Endowed with the Chabauty topology, the space $\Sub(G)$ is compact, and the action of $G$ on $\Sub(G)$ by conjugation is continuous. To every $G$-flow $X$, we can associate the stabilizer map $X \to \Sub(G)$, $x \mapsto G_x$, which is $G$-equivariant. The stabilizer map is always upper semi-continuous (see, e.g., \cite{Glasner2015}), but fails to be continuous in general. This lack of continuity is not just a technical issue, but is an intrinsic property of the flow. For instance, it witnesses the difference between free and topologically free actions (see below). We show that for Gleason complete flows, this defect disappears.

\begin{theorem}
  \label{thm-stab-cont-intro}
	Let $G$ be a locally compact group and let $X$ be a Gleason complete $G$-flow. Then the stabilizer map $X \to \mathrm{Sub}(G)$, $x \mapsto G_x$, is continuous.
\end{theorem}

If $X$ is any $G$-flow, the theorem applies to the Gleason complete flow $\hat X_G$, and shows that taking an irreducible extension of $X$ is enough to resolve the continuity issue of the stabilizer map on $X$. 

As mentioned above, when $G$ is a discrete group, $X$ is Gleason complete if and only if $X$ is extremally disconnected. In that case, Theorem~\ref{thm-stab-cont-intro} is equivalent to saying that the set of fixed points in $X$ of every element $g \in G$ is an open subset of $X$. This is a theorem of Frol\'{\i}k~\cite{Frolik1968}.

Another special case of Theorem~\ref{thm-stab-cont-intro} is a well-known theorem of Veech that the action of a locally compact group on its greatest ambit $\Sam(G)$ is free. One can apply Theorem~\ref{thm-stab-cont-intro} because the greatest ambit is a Gleason complete flow and the free left translation action $G \actson G$ embeds into it densely (cf. Corollary~\ref{c:Veech}). A relativized version of Veech's theorem was considered by Matte Bon and Tsankov in \cite{MatteBon2020}, where it was proved that the stabilizer map for the flow $\Sam(G/H)$ (the Samuel compactification of $G/H$), where $H$ is a closed subgroup of $G$, is continuous. This is again a special case of Theorem~\ref{thm-stab-cont-intro} because the flow $\Sam(G/H)$ is also Gleason complete \cite{Zucker2021}.

As Theorem~\ref{thm-stab-cont-intro} is a common generalization of Frol\'{\i}k's and Veech's theorem, it is perhaps not surprising that its proof mixes ideas from the proofs of both. We also rely on the topometric structure on Gleason complete flows introduced by Zucker~\cite{Zucker2021} (extending a construction of \cite{BenYaacov2017} for $\Sam(G)$), which while being rather simple for locally compact groups, is still useful for us.

\subsection*{Freeness vs topological freeness}

Recall that $G \actson X$ is \df{free} if $G_x$ is trivial for every $x \in X$, and $G \actson X$ is called \df{topologically free} if for every compact $K \sub G$ with $1_G \notin K$, the closed set $\set{x \in X : x \in K \cdot x}$ has empty interior. (When $G$ is second countable, topological freeness is equivalent to saying that there is a dense set of points $x \in X$ such that $G_x$ is trivial.) The difference between freeness and topological freeness is detected by the failure of continuity of the stabilizer map: a topologically free action is free if and only if the stabilizer map is continuous. Also, the property of being topologically free is preserved under irreducible extensions in both directions. Hence the following is a consequence of Theorem~\ref{thm-stab-cont-intro}.

\begin{cor}
	\label{c:intro-top-free}
	Let $G$ be a locally compact group, and let $X$ be a $G$-flow. Then the following are equivalent:
	\begin{enumerate}
		\item $X$ is topologically free;
		\item  $\hat X_G$ is free.
	\end{enumerate}
	In particular, a Gleason complete flow is topologically free if and only if it is free.
\end{cor}

This has the following application. Recall that a $G$-flow is called \df{strongly proximal} if the closure of the $G$-orbit of every Borel probability measure on $X$ contains a Dirac measure. The flow $X$ is called a \df{boundary} if $X$ is minimal and strongly proximal. Every group $G$ admits a boundary $\partial_F G$, unique up to isomorphism, such that every boundary is a factor of $\partial_F G$ \cite[\S III]{Glasner1976}. It is called the \df{Furstenberg boundary} of $G$. By \cite[Lemma 5.2]{Glasner1975} and \cite[Lemma 4.1]{Glasner1976}, the flow $\partial_F G$ is Gleason complete.

\begin{cor} \label{cor-free-bnd}
  For every locally compact group $G$, the stabilizer map is continuous on $\partial_F G$. In particular the following are equivalent:
  \begin{enumerate}
		\item $G$ admits a topologically free boundary;
		\item $G$ acts freely on $\partial_F G$.
	\end{enumerate}
\end{cor}

\begin{proof}
The first assertion follows from the fact that  $\partial_F G$ is Gleason complete and Theorem~\ref{thm-stab-cont-intro}. For the second assertion, if $G$ admits a topologically free boundary $G \actson X$, then the action of $G \actson \partial_F G$ is also topologically free since there is a factor map $\partial_F G \to X$. Since $\partial_F G$ is Gleason complete, Corollary~\ref{c:intro-top-free} implies that $G \actson \partial_F G$ is free. The other direction is clear.
\end{proof}

When $G$ is a discrete group, the equivalence in Corollary \ref{cor-free-bnd} was already known as it follows from \cite{Frolik1968}. Whether this property holds true in a given group $G$ was recently shown to be equivalent to the simplicity of the reduced $\mathrm{C}^\ast$-algebra of $G$ \cite{KK}. It is not known if this equivalence holds more generally for locally compact groups. See \cite[Sec.\ 6]{CKM-typeI} for a discussion of this problem (where the points of $\partial_F G$ where the stabilizer map is continuous are also considered).

\subsection*{Stabilizer flows}

Theorem~\ref{thm-stab-cont-intro} is interesting beyond the case of topologically free actions. Recall that a \df{uniformly recurrent subgroup (URS)} of a locally compact group $G$ is a minimal closed, $G$-invariant subset of $\Sub(G)$ \cite{Glasner2015}. Every minimal $G$-flow $X$ gives rise to a URS of $G$, called the \df{stabilizer URS} associated to $X$, defined as the unique minimal closed $G$-invariant subset of the closure of the image of the stabilizer map in $\mathrm{Sub}(G)$ (Glasner--Weiss~\cite{Glasner2015}). (Although \cite{Glasner2015} makes the standing assumption that $G$ is second countable, this fact holds for every locally compact group and every minimal $G$-flow, see Section \ref{sec:stabilizer-flows} for details.)

Theorem~\ref{thm-stab-cont-intro} allows us to associate a \df{stabilizer flow} to any $G$-flow $X$, without a minimality assumption: we consider the Gleason complete flow $\hat X_G$, and simply take the image of $\hat X_G$ in $\Sub(G)$ by the stabilizer map (cf. Definition~\ref{df:stabilizer-flow}). By definition, the stabilizer flow is an invariant under taking irreducible extensions.
In Section~\ref{sec:stabilizer-flows}, we prove some of its basic properties. We show, in particular, that when $X$ is minimal, the stabilizer flow and the stabilizer URS are equal. 

\begin{cor}
	\label{c:intro-stab-flow-equal-URS}
Let $G$ be a locally compact group, and let $X$ be a minimal $G$-flow. Then the stabilizer URS of $X$ is equal to $\set{G_z : z \in \hat X_G}$.
\end{cor}

In the special case where $X = \hat X_G = \Sam(G/H)$ for some closed subgroup $H \leq G$ belonging to an URS of $G$, Corollary~\ref{c:intro-stab-flow-equal-URS} is equivalent to \cite[Proposition 2.8]{MatteBon2020}, which was used there to prove that every URS of $G$ can be realized as the stabilizer URS of some minimal flow.

\subsection*{Acknowledgments}
We would like to thank Nicolás Matte Bon for interesting discussions on the topic of this work, and Colin Reid and Pierre-Emmanuel Caprace for useful comments on a preliminary version of the article. We are grateful to Andy Zucker for bringing to our attention the inconsistency of the terminology ``maximally highly proximal'' in the case of non-minimal flows and suggesting a change. We are also grateful to the referee for a careful reading of the paper and several remarks and corrections. Work on this paper was partially supported by the ANR project AGRUME (ANR-17-CE40-0026).

\section{The universal irreducible extension of a $G$-flow}
\label{sec:univ-irred-extens}

In this section, we give a new construction of the universal irreducible extension of a $G$-flow $G \actson X$, where $G$ is an arbitrary topological group. The existence of such an extension was proved by Auslander and Glasner \cite{Auslander1977} for minimal flows using an abstract argument and a construction without a minimality assumption, in terms of near-ultrafilters, was given by Zucker~\cite{Zucker2018} for Polish groups and Basso and Zucker~\cite{Basso2021a} for arbitrary topological groups. Our construction is in some sense dual to theirs: instead of constructing the points of $\hat X_G$ directly, we describe the lattice of continuous functions $C(\hat X_G)$ and use an appropriate duality theorem to recover the space.

Before describing the construction, we give an example which illustrates that irreducible and highly proximal extensions are distinct notions, even when the target flow is minimal. 

\begin{example}
  \label{ex:highly-prox}
  Consider the irrational rotation $R \colon \bT \to \bT$ given by $R(x) = x + \alpha$. We construct an extension by doubling the orbit of $0$ as follows. Let $X = \bT \sqcup \set{a_n : n \in \Z}$ and define a metric $d$ on $X$ by setting the distance $d(a_n, n \alpha) = 2^{-|n|}$ and extending it to all of $X$ by taking the shortest path metric, i.e., $d(a_m, a_n) = |n \alpha  - m \alpha| + 2^{-|n|} + 2^{-|m|}$ for $m \neq n$, and $d(a_n, x) = |n \alpha - x| + 2^{-|n|}$ for $x \in \bT$. Then $(X, d)$ is a compact metric space and the map $R$ extends to a homeomorphism of $X$ by setting $R(a_n) = a_{n+1}$ for all $n \in \Z$. The extension map $X \to \bT$ is given by the identity on $\bT$ and $a_n \mapsto n \alpha$. Then one easily checks that this extension is highly proximal but of course it is not irreducible because the points $a_n$ are isolated.
\end{example}

\subsection{The non-archimedean case}
\label{sec:non-archimedean-case}

A Boolean algebra is called \df{complete} if it admits suprema (and infima) of arbitrary subsets. A Boolean algebra $\cB$ is complete iff its Stone space $\tS(\cB)$ is \df{extremally disconnected}, i.e., for every open $U \sub \tS(\cB)$, the set $\cl{U}$ is also open. If $\set{A_i}_{i \in I}$ is a family of clopen sets in $\tS(\cB)$, their supremum in $\cB$ is the clopen set $\cl{\bigcup_i A_i}$.

An open subset $U$ of a topological space $X$ is called \df{regular} if $U = \Int \cl{U}$. The collection $\RO(X)$ of regular open subsets of $X$ forms a complete Boolean algebra with the meet operation $\wedge$ given by the intersection, and complement given by $\neg U = \Int(X \sminus U)$. If $X$ is Baire, $\RO(X)$ can also be viewed as the quotient of the Boolean algebra of Baire measurable subsets of $X$ by the ideal of meager sets. See \cite{Kechris1995}*{Section~8}. We denote by $\hat X$ the Stone space of the algebra $\RO(X)$. If $X$ is compact, there is a natural surjective, continuous map $\ell_X \colon \hat X \to X$ given by
\begin{equation*}
  \set{\ell_X(p)} = \bigcap_{U \in p} \cl{U},
\end{equation*}
where $p$ is viewed as an ultrafilter on $\RO(X)$.

The construction $X \mapsto \hat X$ only depends on the topology of $X$, so if $G$ is a group acting on $X$ by homeomorphisms, it also acts on $\hat X$. If $G$ is a discrete group and $G \actson X$ is a $G$-flow, then $G \actson \hat X$ is also a $G$-flow and it is the universal irreducible extension of $G \actson X$. In particular, if $G$ is discrete, a $G$-flow $X$ is Gleason complete iff $X$ is zero-dimensional and the Boolean algebra of clopen subsets of $X$ is complete.
This follows from the results of Gleason~\cite{Gleason1958}. 

The problem when $G$ has non-trivial topology is that the action $G \actson \hat X$ is not necessarily continuous even if the original action of $G$ on $X$ is. In the case where $G$ is non-archimedean, this is easy to fix. Recall that a topological group $G$ is called \df{non-archimedean} if it admits a basis at $1_G$ consisting of open subgroups. For locally compact groups, by a well known theorem of van Dantzig, being non-archimedean is equivalent to being totally disconnected (or \df{tdlc}, for short).

If $\cB$ is a Boolean algebra on which $G$ acts and $V \leq G$, we will denote by $\cB_V$ the subalgebra of $\cB$ of elements fixed by $V$. Note that if $\cB$ is complete, then $\cB_V$ is complete, too.

If $X$ is a $G$-flow, we let
\begin{equation*}
  \RO(G, X) \coloneqq \bigcup \set{\RO_V(X) :  V \text{ open subgroup of } G}
\end{equation*}
and note that, as a direct limit of Boolean algebras, $\RO(G, X)$ is also a Boolean algebra but that it is not necessarily complete. Note also that $\RO(G, X)$ is invariant under the action of $G$ and that the action $G \actson \RO(G, X)$ is continuous (where $\RO(G, X)$ is taken to be discrete).

\begin{lemma} \label{lem-td-ROG-basis}
Let $G$ be a non-archimedean group and let $G \actson X$ be a $G$-flow. Then the elements of $\RO(G, X)$ form a basis for the topology of $X$. 
\end{lemma}
\begin{proof}
  By regularity of $X$, it suffices to see that for every $x \in U \in \RO(X)$ there exists $U' \in \RO(G, X)$ such that $x \in U' \sub U$. By continuity of the action, there exists an open subgroup $V$ of $G$ and an open subset $U_1 \sub U$ with $x \in U_1$ such that $VU_1 \sub U$. Then $U' = \Int \cl{VU_1}$ works, because $U'$ is $V$-invariant and $U' \sub \Int \cl{U} = U$, since $U$ is regular.
\end{proof}

We denote by $X^\ast_G$ the Stone space of $\RO(G, X)$. The action of $G$ on $X^\ast_G$ is continuous. Note that $X^\ast_G$, being the Stone space of a Boolean algebra, is zero-dimensional. 

\begin{prop}
  \label{p:univ-hp-non-archim}
  Let $G$ be a non-archimedean group and let $G \actson X$ be a $G$-flow. Then  $G \actson X^\ast_G$ is the universal irreducible extension of $X$. 
\end{prop}

\begin{proof}
  We denote by $\pi \colon \hat X \to X^\ast_G$ the dual map of the inclusion $\RO(G, X) \sub \RO(X)$ and note that $\pi$ is continuous and $G$-equivariant. By Lemma \ref{lem-td-ROG-basis}, if two elements of $\hat X$ have the same image by $\pi$, then they have the same image under the map $\ell \colon \hat X \to X$. Hence there is a continuous $G$-equivariant map $\ell_{G} \colon X^\ast_G \to X$ such that $\ell_{G} \circ \pi = \ell$.  The map $\ell_{G} \colon X^\ast_G \to X$ is irreducible because $\ell$ is. If $Y \to X$ is an irreducible extension of $X$, then $\hat Y = \hat X$. Thus $\RO(X) = \RO(Y)$ and $\RO(G, X) = \RO(G, Y)$. In particular, $Y$ is a factor of $X^\ast_G = Y^\ast_G$.
\end{proof}

By continuity of the $G$-action on $X$, we have $\Clopen(X) \subseteq \RO(G, X)$, where $\Clopen(X)$ is the subalgebra of $\RO(X)$ consisting of clopen subsets of $X$. That this inclusion is an equality actually characterizes Gleason complete flows for non-archimedean groups.

\begin{cor}
  \label{c:char-mhp}
  Let $G$ be a non-archimedean group and let $G \actson X$ be a $G$-flow. Then the following are equivalent:
  \begin{enumerate}
  \item \label{item-mhp-td} $X$ is Gleason complete;
  \item \label{item-RO-td} $\RO(G, X) = \Clopen(X)$.
  \end{enumerate}
\end{cor}
\begin{proof}
\ref{item-mhp-td} $\Rightarrow$ \ref{item-RO-td} follows from Proposition~\ref{p:univ-hp-non-archim}. Note that \ref{item-RO-td} implies that $X$ is zero-dimensional in view of Lemma \ref{lem-td-ROG-basis}, so the implication \ref{item-RO-td} $\Rightarrow$ \ref{item-mhp-td} also follows from Proposition~\ref{p:univ-hp-non-archim}. 
\end{proof}


\subsection{The general case}
\label{sec:general-case}

When $G$ is a general topological group, one cannot hope to construct the universal irreducible extension as the Stone space of a Boolean algebra: for example, if $G$ is connected, then all of its minimal flows are connected and have no non-trivial clopen sets. So for the general case, we employ Riesz spaces instead of Boolean algebras. 

Recall that a \df{Riesz space} is an ordered real vector space, which is a \df{lattice} for the ordering, i.e., all pairs of elements $a, b$ have a least upper bound $a \vee b$ and a greatest lower bound $a \wedge b$. A Riesz space $\cL$ is called \df{archimedean} if there exists a \df{unit} $\bOne \in \cL$ such that for every $a \in \cL$, there exists $n \in \N$ with $a \leq n  \bOne$. A unit also naturally defines the \df{uniform norm}:
\begin{equation*}
  \nm{a} \coloneqq \inf \set{r \in \R : |a| \leq r \bOne},
\end{equation*}
where, as usual, $|a| = a \vee (-a)$.

A natural example of an archimedean Riesz space is the collection of real-valued continuous functions $C(X)$ on a compact space $X$ with the usual lattice operations and unit the constant function $\bOne$. Then the uniform norm coincides with the $\sup$ norm. The Yosida representation theorem, which we recall below, states that in fact every archimedean Riesz space complete in the uniform norm is of this form.

For every archimedean Riesz space $\cL$ with a unit $\bOne$ (and equipped with the uniform norm), we can consider its \df{spectrum}:
\begin{equation*}
  \tS(\cL) = \set{x \in \cL^* : x(a \vee b) = x(a) \vee x(b) \text{ for all }a, b \in \cL \And x(\bOne) = 1}.
\end{equation*}
$\tS(\cL)$ is a compact space if equipped with the weak$^*$ topology and we have a map $\Gamma \colon \cL \to C(\tS(\cL))$ defined by
\begin{equation*}
  \Gamma(a)(x) = x(a).
\end{equation*}
$\Gamma$ is clearly a contractive homomorphism and, in fact, it is an isometric isomorphism (see \cite{Jonge1977}*{Section~13}). 

Let $G$ be a topological group, let $E$ be a Banach space and let $G \actson E$ be an action by isometric isomorphisms. We will say that an element $\phi \in E$ is \df{$G$-continuous} if the map $G \to E, g \mapsto g \cdot \phi$ is norm-continuous.

\begin{lemma}
  \label{l:continuous-action}
  Let $G$ be a topological group, let $X$ be a compact space and let $G \actson X$ be an action by homeomorphisms. Then the following are equivalent:
  \begin{enumerate}
  \item $G \actson X$ is a $G$-flow (that is, the action is jointly continuous);
  \item Every function $\phi \in C(X)$ is $G$-continuous for the induced action $G \actson C(X)$.
  \end{enumerate}
\end{lemma}
\begin{proof}
  \begin{cycprf}
  \item [\impnext] This is obvious.
  \item [\impfirst] Let $U \sub X$ be open and let $x_0 \in U$. Our goal is to find an open $V \ni 1_G$ and an open $W \ni x_0$ such that $V \cdot W \sub U$. Let $W \ni x_0$ be open such that $\cl{W} \sub U$. By Urysohn's lemma, there exists $\phi \in C(X)$ with $\phi|_{\cl{W}} = 1$ and $\phi|_{X \sminus U} = 0$. As $\phi$ is $G$-continuous, there exists $V \ni 1_G$ such that for every $v \in V$, $\nm{v^{-1} \cdot \phi - \phi} < 1/2$. This implies that $V \cdot W \sub U$.
  \end{cycprf}
\end{proof}

Next we will describe the universal irreducible extension of a $G$-flow $G \actson X$. Let $\cB(X)$ denote the Riesz space of bounded Borel functions on $X$ with unit the constant function $\bOne$ and let $\cM$ be the ideal given by:
\begin{equation*}
  \cM = \set[\big]{\phi \in \cB(X) : \set{x \in X : \phi(x) \neq 0} \text{ is meager}}.
\end{equation*}
The ideal $\cM$ allows us to define the essential supremum seminorm on $\cB(X)$ by
\begin{equation*}
  \nm{\phi}_\cM = \inf \set[\big]{r \in \R : \set{x \in X : |\phi(x)| > r } \text{ is meager} }.
\end{equation*}
Denote $B(X) \coloneqq \cB(X)/\cM$ and note that $\nm{\cdot}_\cM$ descends to a norm on $B(X)$. Let $\hat X$ be the spectrum of $B(X)$. It can naturally be identified with the Stone space of the Boolean algebra $\RO(X)$ (see \cite{Jonge1977}*{Section~14}). The space $B(X)$ has also been considered before in a dynamical context by Keynes and Robertson in \cite{Keynes1968}. 

We let $B_G(X)$ denote the set of $G$-continuous elements of $B(X)$. We note that $B_G(X)$ is a closed subspace of $B(X)$ which is also closed under the lattice operations, so we can define $\hat X_G \coloneqq \tS(B_G(X))$. Because $C(\hat X_G) \cong B_G(X)$, it follows from Lemma~\ref{l:continuous-action} that $G \actson \hat X_G$ is a $G$-flow. There is a natural injective map $C(X) \to B(X)$, whose image is, by virtue of Lemma~\ref{l:continuous-action}, contained in $B_G(X)$. Slightly abusing notation, we will identify $C(X)$ with its image in $B_G(X)$. The inclusions $C(X) \sub B_G(X) \sub B(X)$ translate to factor maps $\hat{X} \to \hat X_G \to X$. We have the following.
 
\begin{prop}
  \label{p:mhp-extension-general}
  Let $G$ be a topological group and let $G \actson X$ be a $G$-flow. Then the flow $G \actson \hat X_G$ is the universal irreducible extension of $G \actson X$. In particular, $X$ is Gleason complete if and only if the natural injection $C(X) \to B_G(X)$ is a bijection.
\end{prop}
\begin{proof}
First, as $\hat X_G$ is a factor of $\hat X$, it is clear that the extension $\hat X_G  \to X$ is irreducible. If $G \actson Y$ is an irreducible extension of $G \actson X$, then by the universal property of $\hat X$, there exists an embedding $\iota \colon C(Y) \to C(\hat X) = B(X)$ (see above). It follows from Lemma~\ref{l:continuous-action} that every $\phi \in C(Y)$ is $G$-continuous, so $\iota(\phi)$ is also $G$-continuous. Therefore $\iota(C(Y)) \sub B_G(X)$ and this gives a factor map $\hat X_G \to Y$.
\end{proof}

\begin{cor} \label{cor-Gleason complete-localcondition}
Let $G$ be a topological group and let $L$ be an open subgroup of $G$. If $G \actson X$ is a $G$-flow, then $X$ is Gleason complete as a $G$-flow if and only if it is Gleason complete as an $L$-flow.
\end{cor}
\begin{proof}
We have $B_G(X) = B_L(X)$ since $L$ is open, so the statement follows from Proposition   \ref{p:mhp-extension-general}. 
\end{proof}

It is proved in \cite{Auslander1977} that for minimal flows $X$, the correspondence $X \mapsto \hat X_G$ is functorial. Our description of $\hat X_G$ suggests the correct formulation of this result for general flows. Recall that a continuous map $\phi \colon X \to Y$ is called \df{category-preserving} if $\phi^{-1}(A)$ is nowhere dense for any nowhere dense $A \sub Y$. Every homomorphism between minimal flows is category-preserving. Indeed, if $\phi \colon X \to Y$ is a factor map between minimal flows and $U \sub X$ is open, non-empty, then finitely many translates of $\phi(U)$ cover $Y$, so $\phi(U)$ must be somewhere dense. Also, every irreducible map between compact spaces is category-preserving.

\begin{prop}
  \label{p:functor}
  The correspondence $X \mapsto \hat X_G$ is a functor from the category of $G$-flows with morphisms category-preserving $G$-maps to the category of Gleason complete flows.
\end{prop}
\begin{proof}
  Let $\phi \colon X \to Y$ be a category-preserving homomorphism of $G$-flows. Then $\phi^{-1}(A)$ is meager for every meager set $A \sub Y$, so we obtain a dual homomorphism of Riesz spaces $\phi^* \colon B(Y) \to B(X)$ given by $\phi^*([f]) = [f \circ \phi]$, where $f \in \cB(Y)$ and $[f]$ denotes its equivalence class in $B(Y)$. The image of $B_G(Y)$ is contained in $B_G(X)$, so by the duality theorem, this gives us a map $\hat X_G \to \hat Y_G$.
\end{proof}


\section{Characterizations of Gleason complete flows} \label{sec-carac-mhp}

Starting from this section, $G$ will denote a locally compact group. A \df{pseudo-norm} on $G$ is a continuous function $\nm{\cdot} \colon G \to \R_+$ satisfying:
\begin{itemize}
\item $\nm{1_G} = 0$;
\item $\nm{g^{-1}} = \nm{g}$ for all $g \in G$;
\item $\nm{gh} \leq \nm{g} + \nm{h}$ for all $g, h \in G$.
\end{itemize}

We denote by $B_r$ the set of elements $g \in G$ such that $\nm{g} < r$, and we also let $\bar{B}_{r}$ be the set of elements $g \in G$ such that $\nm{g} \leq r$. We say that $\nm{\cdot}$ is \df{proper} if $\bar{B}_{r}$ is compact for all $r$. We say that a pseudo-norm $\nm{\cdot}$ is a \df{norm} if it satisfies that the only element $g$ with $\nm{g} = 0$ is $1_G$. A norm is called \df{compatible} if it induces the topology of $G$. Note that any group $G$ that admits a proper pseudo-norm must be $\sigma$-compact (because $G = \bigcup_{n \in \N} \bar{B}_n$).

Every pseudo-norm induces a right-invariant pseudo-metric $d_r$ on $G$ defined by
\begin{equation}
  \label{eq:right-inv-metric}
  d_r(g, h) = \nm{gh^{-1}}.
\end{equation}

We will say that a pseudo-norm $\nm{\cdot}$ is \df{normal} if for every $g \in G$, the conjugation by $g$ is a uniformly continuous map of the pseudo-metric space $(G,d_r)$. In particular, the kernel $\set{g \in G : \nm{g} = 0}$ of a normal pseudo-norm is a normal subgroup of $G$.

\begin{prop}
	\label{prop-exist-contin-pseudonorm}
	Let $G$ be a $\sigma$-compact locally compact group, and let $V$ be an open neighborhood of $1_G$. Then there exists a  proper, normal pseudo-norm on $G$ and $r > 0$ such that $B_r \subseteq V$.
\end{prop}

\begin{proof}
	Choose and open neighborhood $W$ of $1_G$ such that $W^2 \subseteq V$. Since $G$ is $\sigma$-compact, theorems of Kakutani--Kodaira (\cite[Theorem 8.7]{Hewitt1979}) and Struble~\cite{Struble1974} ensure that there exists a compact normal subgroup $K$ of $G$ with $K \subseteq W$ such that $G/K$ admits a compatible and proper norm $\nm{\cdot}_{G/K}$. If we let $\nm{g} = \nm{gK}_{G/K}$, then $\nm{\cdot}$ is a pseudo-norm on $G$ that is proper. Moreover since the image of $W$ in $G/K$ is an open neighborhood of the identity in $G/K$ and $\nm{\cdot}_{G/K}$ is compatible, there is $r > 0$ such that $\nm{gK}_{G/K} < r$ implies $gK \in WK$. Hence $\nm{g} < r$ implies $g \in V$. Normality is clear since $\nm{\cdot}_{G/K}$ induces the topology on $G/K$. 
\end{proof}

Let $\nm{\cdot}$ be some fixed proper pseudo-norm on $G$. If $G \actson X$ is a $G$-flow, we can define a pseudo-metric $\dtp$ on $X$ by
\begin{equation}
  \label{eq:dfn-dtp}
\dtp(x, y) = \inf \set{\nm{g} : g \in G, g \cdot x = y}.
\end{equation}
If $x$ and $y$ are not in the same orbit, then $\dtp(x, y) = \infty$.

Note that since $\nm{\cdot}$ is proper, $\dtp$ is always lower semi-continuous for the compact topology $\tau$ on $X$. Recall that a real-valued function $f$ is called \df{lower semi-continuous (lsc)} if for every real number $r$ the set $\left\lbrace f > r \right\rbrace$ is open. It is \df{upper semi-continuous (usc)} if $\left\lbrace f < r \right\rbrace$ is open.

When $G$ is metrizable, we can work throughout with a fixed compatible, proper norm on $G$, and then $\dtp$ is a metric on $X$ that refines the topology $\tau$, i.e., $(X, \tau, \dtp)$ is a \df{compact topometric space} in the sense of \cite{BenYaacov2013b}. In general, one can work with a \df{topouniform spaces} as is done in \cite{Basso2021a}, but we will not need this here. In the case where $G$ is Polish, locally compact, the topometric space above is the same as the one considered by Zucker~\cite{Zucker2021}. The metric $\dtp$ also provides a convenient way to express $G$-continuity: a function $f \colon X \to \R$ is $G$-continuous iff it is uniformly continuous as a function on the metric space $(X, \dtp)$.

The following characterization of Gleason complete flows is the main theorem of this section. Note that because every locally compact group admits an open, $\sigma$-compact subgroup (for example, the subgroup generated by any compact neighborhood of the identity), the condition in the theorem is not restrictive. If $G$ is already $\sigma$-compact, one can simply take $L = G$ below.
\begin{theorem}
	\label{th:char-distance}
	Let $G$ be a locally compact group and let $L$ be a $\sigma$-compact, open subgroup of $G$. If $G \actson X$ is a $G$-flow, the following are equivalent:
	\begin{enumerate}
    \item \label{item-defGleason complete} $X$ is Gleason complete;
    \item \label{item-car-loc} $V \overline{U}$ is open for every open neighborhood $V$ of $1_G$ and open subset $U$ of $X$;
    \item \label{item-cont-dist} for every proper pseudo-norm on $L$ (with the associated pseudo-metric $\dtp$ for the action $L \actson X$) and for every open subset $U$ of $X$, the function $X \to \R \cup \left\lbrace \infty \right\rbrace $, $x \mapsto \dtp(x, \cl{U})$, is continuous.
	\end{enumerate}

When $G$ is a tdlc group, these are also equivalent to:
\begin{enumerate}[resume]
	\item \label{item-RO-tdlc} $\RO(G, X) = \Clopen(X)$;
  \item \label{item-complete-tdlc} $X$ is zero-dimensional and for every compact open subgroup $V$ of $G$, the Boolean algebra $\Clopen_V(X)$ is complete.
\end{enumerate}
\end{theorem}

 Before going further, we make a few comments. First note that it follows in particular that statement \ref{item-cont-dist} holds for some $L$ iff it holds for every $L$. The equivalence between \ref{item-defGleason complete} and  \ref{item-car-loc} is already contained in \cite{Zucker2021} (up to the observation that when $G$ is locally compact, Definition 3.1 from \cite{Zucker2021} can be restated as in \ref{item-car-loc}). Here we provide an alternative proof of that equivalence. The proof of \ref{item-car-loc} $\Rightarrow$ \ref{item-defGleason complete} follows arguments close to \cite{Gleason1958}, while the proof of the converse (which goes through \ref{item-cont-dist}) uses the characterization of Gleason complete flows given in Proposition  \ref{p:mhp-extension-general}. 

We need some preliminaries. 

\begin{lemma} \label{lem-dist-closed-open}
	\label{l:lsc-usc}
Let $G$ be a locally compact, $\sigma$-compact group with a proper pseudo-norm $\nm{\cdot}$. Let $G \actson X$ be a $G$-flow and let $\dtp$ be defined as above. Then the following hold:
	\begin{enumerate}
		\item \label{i:l:lsc} If $F \sub X$ is closed, the function $x \mapsto \dtp(x, F)$ is lsc.
		\item \label{i:l:usc} If $U \sub X$ is open, the function $x \mapsto \dtp(x, U)$ is usc.
	\end{enumerate}
\end{lemma}
\begin{proof}
	\ref{i:l:lsc} Let $A = \set{x : \dtp(x, F) \leq r}$. Let $x$ be a limit point of $A$ and let $(x_i, \eps_i)_i$ be a net in $A \times \R^+$ converging to $(x, 0)$. Let $y_i \in F$ be such that $\dtp(x_i, y_i) < r + \eps_i$. By passing to a subnet, we may assume that $y_i \to y \in F$. Then taking limits and using the fact that $\dtp$ is lsc, we obtain that $\dtp(x, y) \leq r$.
	
	\ref{i:l:usc} Let $r > 0$, and let $V$ be the open ball around $1_G$ of radius $r$. Then
	\begin{equation*}
		\dtp(x, U) < r \iff V \cdot x \cap U \neq \emptyset,
	\end{equation*}
	which is an open condition.
\end{proof}

\begin{remark}
	\label{rem:lemma-usc-lsc}
	In fact, Lemma~\ref{l:lsc-usc} does not need $G$ to be locally compact (with the appropriate definition of $\dtp$ in the general case, see \cite{Zucker2021}). The proof of \ref{i:l:lsc} works as above and \ref{i:l:usc} is \cite{Zucker2021}*{Theorem~4.8} and it is harder.
\end{remark}

\begin{lemma} \label{lem-car-loc-sub}
	Let $X$ be a $G$-flow. Then $X$ satisfies condition \ref{item-car-loc} of  Theorem 	\ref{th:char-distance} if and only if for every open neighborhood $V \ni 1_G$ and open subset $U \sub X$, there exists an open neighborhood $V' \ni 1_G$  with $V' \sub V$ such that $V' \overline{U}$ is open. 	 
\end{lemma}
\begin{proof}
	We only have to prove the implication from right to left. Suppose that the property in the statement holds, and let $U$ be an open subset of $X$ and $V$ an open neighborhood of $1_G$. For every $g$ in $V$ one can find $V'_g$ an open neighborhood of $1_G$ such that $g V'_g$ is contained in $V$ and $V'_g \overline{U}$ is open. Writing $V = \bigcup_{g \in V} V'_g$, we then have $V \overline{U} =  \bigcup_{g \in V} V'_g \overline{U}$, which is this thus open.
\end{proof}

\begin{lemma} \label{lem-open-intersect}
	Let $X$ be a $G$-flow that satisfies condition \ref{item-car-loc} of  Theorem 	\ref{th:char-distance}. Then for all open subsets $U_1, U_2 \sub X$, we have $\overline{U_1} \cap \overline{U_2} \neq \emptyset$ if and only if $V U_1 \cap U_2 \neq \emptyset$ for every open neighborhood $V \ni 1_G$.
\end{lemma}

\begin{proof}
	Suppose $\overline{U_1} \cap \overline{U_2} \neq \emptyset$, and let $V$ be an open neighborhood of $1_G$. Then clearly $V \overline{U_1} \cap \overline{U_2} \neq \emptyset$. Since $V \overline{U_1}$ is open, this implies that $V \overline{U_1} \cap U_2 \neq \emptyset$. That condition is equivalent to $\overline{U_1} \cap V^{-1} U_2 \neq \emptyset$ and hence implies that $U_1 \cap V^{-1} U_2 \neq \emptyset$. So $V U_1 \cap U_2 \neq \emptyset$, as desired. The reverse implication is a general fact that follows from continuity of the $G$-action.  
\end{proof}

Recall that a subalgebra $A$ of a Boolean algebra $B$ is \df{dense} if for every non-zero element in $B$ there is a non-zero element in $A$ that is smaller. We recall the following (see \cite{Koppelberg1989}*{Theorem~4.19}).

\begin{lemma} \label{lem-boolean}
	Let $A$ be a dense subalgebra of a Boolean algebra $B$. If $A$ is complete, then $A = B$.
\end{lemma}

Before starting the proof of Theorem \ref{th:char-distance}, we also introduce some notation. If $\pi \colon Y \to X$ is a continuous map between topological spaces and $U \sub Y$ is open, we denote by $\pi_*(U)$ the \df{fiber image} of $U$:
\begin{equation*}
  \pi_*(U) \coloneqq \set{x \in X : \pi^{-1}(x) \sub U}.
\end{equation*}
If the space $Y$ is compact, the set $\pi_*(U)$ is always open, and if $\pi$ is irreducible, $\pi_*(U)$ is non-empty for any non-empty $U$.

\begin{proof}[Proof of Theorem 	\ref{th:char-distance}]
  \ref{item-car-loc} $\Rightarrow$ \ref{item-defGleason complete}. Let $\pi \colon Y \to X$ be an irreducible extension. We shall prove that $\pi$ is injective. Suppose for a contradiction that there exist distinct points $y_1,y_2$ in $Y$ with the same image $x$ in $X$. Then one can find an open $V \ni 1_G$ and open subsets $O_1,O_{2} \sub Y$ such that $y_1 \in O_1$, $y_2 \in O_{2}$ and $V O_1 \cap O_{2} = \emptyset$. The irreducibility of $\pi$ implies that $\pi(O) \sub \cl{\pi_*(O)}$ for any open $O \sub Y$. Indeed, if not, there is $y \in O$ and an open $W \ni \pi(y)$ disjoint from $\pi_*(O)$. By irreducibility, $\pi^{-1}(W) \cap O$ contains a fiber, whose image must be in $\pi_*(O)$, contradiction. Thus the sets $\overline{\pi_*(O_1)}$ and $\overline{\pi_*(O_2)}$ both contain $x$.
  Hence by the assumption \ref{item-car-loc} and Lemma~\ref{lem-open-intersect}, we have $V \pi_*(O_1) \cap \pi_*(O_2) \neq \emptyset$. Since $V \pi_*(O_1) = \pi_*(V O_1)$, we deduce that $V O_1$ and $O_{2}$ intersect each other, which is a contradiction. 

\ref{item-cont-dist}  $\Rightarrow$ \ref{item-car-loc}. Let $V$ be an open neighborhood of $1_G$, and $U$ an open subset of $X$. By Lemma  \ref{lem-car-loc-sub}, upon replacing $V$ by $V \cap L$ we can assume that $V$ is contained in $L$. Applying Proposition \ref{prop-exist-contin-pseudonorm}, we can find a continuous proper pseudo-norm on $L$ and $r > 0$ such that $B_{r}$ is contained in $V$. If $\dtp$ is the pseudo-metric on $X$ associated to this pseudo-norm, by assumption, the function $f(x) \coloneqq \dtp(x, \cl{U})$ is continuous. So $B_{r} \overline{U} = \left\lbrace f < r \right\rbrace$ is open. Since $B_{r} \subseteq V$ and $V$ was arbitrary, Lemma  \ref{lem-car-loc-sub} ensures that \ref{item-car-loc} holds.

 \ref{item-defGleason complete} $\Rightarrow$ \ref{item-cont-dist}. Fix a continuous proper pseudo-norm on $L$, and an open subset $U$ of $X$.  Let $\phi_0(x) = \dtp(x, \cl{U})$ and $\phi_1(x) = \dtp(x, U)$. We have that $\phi_0$ is lsc and $\phi_1$ is usc by Lemma \ref{lem-dist-closed-open}. Moreover, $\phi_0 \leq \phi_1$, and both $\phi_0$ and $\phi_1$ are $\dtp$-contractive (meaning that $|\phi_i(x) - \phi_i(y)| \leq \dtp(x,y)$ for all $x,y$). First we show that the set $\set{\phi_0 < \phi_1}$ is meager. Note that
 \begin{equation*}
 \set{\phi_0 < \phi_1} = \bigcup_{q_1 < q_2 \in \Q} \set{\phi_0 \leq q_1 < q_2 \leq \phi_1}
 \end{equation*}
 and each set in the union is closed. So if $\set{\phi_0 < \phi_1}$ is non-meager, there exist $q_1 < q_2$ such that $\set{\phi_0 \leq q_1 < q_2 \leq \phi_1}$ has non-empty interior $W$. The set $\set{x : \dtp(x, W) < q_2}$ is open and intersects $\cl{U}$, so it must intersect $U$. So there exist $x \in U$, $y \in W$ with $\dtp(x, y) < q_2$, which contradicts the definition of $W$.
 
 Now for $r > 0$, set $\phi_{0,r}= \min(\phi_0,r)$ and  $\phi_{1,r} = \min(\phi_1,r)$. The functions $\phi_{0,r},\phi_{1,r}$ are bounded and remain $\dtp$-contractive (hence $L$-continuous). As $X$ is Gleason complete as a $G$-flow by assumption, it is also Gleason complete as a $L$-flow by Corollary \ref{cor-Gleason complete-localcondition}. So by Proposition \ref{p:mhp-extension-general}, there exists a continuous function $\theta$ on $X$ such that $\phi_{0,r} = \phi_{1,r} = \theta$ on a comeager set. As the sets $\set{\theta < \phi_{0,r}}$ and $\set{\theta > \phi_{1,r}}$ are open, they must be empty, and we must have that $\phi_{0,r}  \leq \theta \leq \phi_{1,r}$. We claim that $\theta$ is $\dtp$-contractive. If not, there exist $x \in X$ and $g \in L$ such that $|\theta(x) - \theta(g \cdot x)| > \nm{g}$. However, the set $\set{x : |\theta(x) - \theta(g \cdot x)| > \nm{g}}$ is open, so as $\theta = \phi_{0,r}$ on a comeager set, there exists $x$ such that $|\phi_{0,r}(x) - \phi_{0,r}(g \cdot x)| > \nm{g}$, contradiction. Note that $\theta^{-1}(0) \supseteq U$, so by continuity, $\theta^{-1}(0) \supseteq \cl{U}$. As $\theta$ is $\dtp$-contractive and $\theta = 0$ on $\cl{U}$, for every $x \in X$, we have
 \begin{equation*}
 \theta(x) \leq \inf_{y \in \cl{U}} \dtp(x, y) = \dtp(x, \cl{U}) =  \phi_0(x).
 \end{equation*}
 Since $\theta \leq r$, this shows that $\theta \leq \phi_{0,r}$, and hence $\theta = \phi_{0,r}$. So $\phi_{0,r}$ is continuous. Since $r$ is arbitrary, it follows that $\phi_{0}$ is continuous. 

We now assume $G$ is a tdlc group. The equivalence between \ref{item-defGleason complete} and \ref{item-RO-tdlc} follows from Corollary~\ref{c:char-mhp}. Recall in particular that these imply that $X$ is zero-dimensional. Hence the fact that  \ref{item-RO-tdlc} implies \ref{item-complete-tdlc} is clear since $\RO_V(X)$ is always complete. It remains to see that \ref{item-complete-tdlc} implies \ref{item-RO-tdlc}. To that end, let $V$ be a compact open subgroup of $G$. We want to see that $\RO_V(X) = \Clopen_V(X)$. We claim that $\Clopen_V(X)$ is a dense subalgebra of $\RO_V(X)$. Indeed, if $U$ is a non-empty element of $\RO_V(X)$, then we can find a non-empty clopen subset $U_1$ inside $U$ since $X$ is zero-dimensional. Since $V$ is compact and open, the stabilizer of $U_1$ has finite index in $V$, so that $VU_1$ is a union of finitely many clopen subsets, and hence is clopen. Moreover $VU_1 \subseteq U$ since $U$ is $V$-invariant. Hence $\Clopen_V(X)$ is dense in $\RO_V(G, X)$. Since we make the assumption that $\Clopen_V(X)$ is complete, Lemma~\ref{lem-boolean} implies that $\RO_V(X) = \Clopen_V(X)$, as desired. 
\end{proof}

 Compare the next corollary with \cite{BenYaacov2017}*{Lemma~2.4}.
\begin{cor}
  \label{c:d-closure}
  Let $G$ be a locally compact, $\sigma$-compact group equipped with a proper pseudo-norm $\nm{\cdot}$, and let $G \actson X$ be a Gleason complete flow. Then for $U_1, U_2 \sub X$ open,
  \begin{equation*}
    \dtp(\cl{U_1}, \cl{U_2}) = \dtp(U_1, U_2).
  \end{equation*}
\end{cor}
\begin{proof}
  Suppose that $\dtp(\cl{U_1}, \cl{U_2}) < r$. Consider the set $\set{x : \dtp(x, \cl{U_2}) < r}$. By Theorem~\ref{th:char-distance}, it is open and it intersects $\cl{U_1}$, so it intersects $U_1$. Let $W \sub \set{x : \dtp(x, \cl{U_2}) < r}$ be open, non-empty with $\cl{W} \sub U_1$. Then by continuity of the function $\dtp(\cdot, \cl{W})$, there exists $x \in U_2$ with $\dtp(x, \cl{W}) < r$. So $\dtp(U_1, U_2) < r$.
\end{proof}


\section{Continuity of the stabilizer map}
\label{sec:cont-stab-map}

Let $Y$ be a locally compact space and let $2^Y$ denote the space of closed subsets of $Y$. The \df{Chabauty topology} on $2^Y$ is given by the subbasis of sets of the form
\begin{equation*}
O_K = \set{F \in 2^Y : F \cap K = \emptyset} \quad \And \quad  O^U = \set{F \in 2^Y : F \cap U \neq \emptyset}
\end{equation*}
with $K \sub Y$ compact and $U \sub Y$ open. The space $2^Y$ equipped with this topology is compact. A map $\phi \colon X \to 2^Y$ is \df{upper semi-continuous} if $\phi^{-1}(O_K)$ is open for every compact subset $K$ of $Y$ and it is \df{lower semi-continuous} if $\phi^{-1}(O^U)$ is open for every open subset $U$ of $Y$. 

If $G$ is a locally compact group, the set $\Sub(G)$ of closed subgroups of $G$ is closed in $2^G$, and hence, a compact space. Moreover, the conjugation action of $G$ on $\Sub(G)$ is continuous. 

\begin{defn}
Let $X$ be a $G$-flow. For $x \in X$, let $G_x$ denote the stabilizer of $x$. The map $\Stab \colon X \to \Sub(G)$ defined by $\Stab(x) = G_x$ is called the \df{stabilizer map} associated to the flow $X$.
\end{defn}

It is easy to see that for every $G$-flow, the stabilizer map is $G$-equivariant and upper semi-continuous (see, e.g., \cite{Glasner2015}). It is also well-known that in general it is not continuous. The main theorem of the paper is the following.

\begin{theorem} \label{thm-stab-cont-mhp}
	Let $G$ be a locally compact group and let $X$ be a Gleason complete $G$-flow. Then the stabilizer map $X \to \mathrm{Sub}(G)$, $x \mapsto G_x$ is continuous.
\end{theorem}

The remainder of this section is devoted to the proof of the theorem. Let $\nm{\cdot} \colon G \to \R_+$  be a pseudo-norm on $G$. We recall that $B_r$ denotes the set of elements $g \in G$ such that $\nm{g} < r$; we also let $\bar{B}_{r}$ be the set of elements $g \in G$ such that $\nm{g} \leq r$. Recall that if $X$ is a $G$-flow, we associated, by the equation \eqref{eq:dfn-dtp}, a pseudo-metric $\dtp$ on $X$. The following is the main lemma.



\begin{lemma} \label{l:construct-function}
  Let $G$ be a locally compact, $\sigma$-compact group and let $\nm{\cdot}$ be a continuous, proper, normal pseudo-norm on $G$. Let $X$ be a Gleason complete $G$-flow, let $g \in G$ and $r > 0$. Then there exist $n \geq 1$ and a continuous function $\phi \colon X \to \R^n$ such that for all $x \in X$
  \begin{equation}
    \label{eq:prop-phi}
  \partial(g \cdot x,x) > r \implies \nm{\phi(g \cdot x) - \phi(x)}_{\infty} \geq r/3.
\end{equation}
\end{lemma}
\begin{proof}
Since $\nm{\cdot}$ is normal, $g$ and $g^{-1}$ are uniformly continuous as self-maps of $(X, \dtp)$. So let $\delta < r/3$ and $\delta'$ be such that
\begin{gather*}
\forall x, y \in X \quad \dtp(x, y) < \delta \implies \dtp(g \cdot x, g \cdot y) < r/3 \\
\forall x, y \in X \quad \dtp(g \cdot x, g \cdot y) < \delta' \implies \dtp(x, y) < \delta.
\end{gather*}

Since $\nm{\cdot}$ is continuous and proper, one can find $g_1, \ldots,g_\ell$ such that $\bar{B}_{2r}$ is contained in $\bigcup_{i=1}^\ell g_i B_{\delta/2}$. By the pigeonhole principle, this implies that a ball of radius $2r$ in $(X, \dtp)$ cannot contain more than $\ell$ points which are pairwise at least $\delta$ apart. That is, for every $x, x_1, \ldots, x_{\ell +1} \in X$ such that $\partial(x,x_i) \leq 2r$ for all $i$, there are $i \neq j$ such that $\partial(x_i,x_j) < \delta$. Similarly, there is $k \in \N$ such that for all $x, x_1, \ldots, x_{k +1} \in X$ such that $\partial(x,x_i) \leq 2r$ for all $i$ there are $i \neq j$ such that $\partial(x_i,x_j) < \delta'$. Set $n = k+ \ell + 1$.

Set $M_r = \set{x : \dtp(g \cdot x, x) > r}$. We will construct open sets $U_1, \ldots, U_n \sub X$ with the following properties:
\begin{enumerate}
	\item \label{i:pf:stab:1} the closure of $\bigcup_i B_\delta U_i$ contains $M_r$;
	\item \label{i:pf:stab:2} $\dtp(U_i, U_j) \geq \delta$ for $i \neq j$;
	\item \label{i:pf:stab:3} $\dtp(g \cdot U_i, U_i) \geq r$ for all $i$.
\end{enumerate}

Once the construction is completed, we finish the proof as follows. We set
\begin{equation*}
\phi_i(x) = \min(\dtp(x, \cl{U_i}), r)
\end{equation*}
and $\phi = (\phi_i)_i$. By Theorem~\ref{th:char-distance}, $\phi$ is continuous. To see that $\phi$ satisfies the conclusion, in view of \ref{i:pf:stab:1} it is enough to see that $|| \phi(g \cdot x) - \phi(x) ||_{\infty} \geq r/3$ for every $x$ in $\bigcup_i B_\delta U_i$. So let $x \in B_\delta U_i$ and let $y \in U_i$ be such that $\dtp(x, y) < \delta$. Then $\dtp(g \cdot x, g \cdot y) < r/3$ and using Corollary~\ref{c:d-closure}, we obtain

\begin{equation*}
  \begin{split}
    \partial(g \cdot x, \cl{U_i}) &\geq \dtp(g \cdot y, \cl{U_i}) - \dtp(g \cdot y, g \cdot x) \\
    & \geq \dtp(g \cdot \cl{U_i}, \cl{U_i}) - r/3  \\
    & = \dtp(g \cdot U_i, U_i) - r/3 \geq 2r/3.
  \end{split}
\end{equation*}
So
\[ || \phi(g \cdot x) - \phi(x) ||_{\infty}  \geq \phi_i(g \cdot x) - \phi_i(x) \geq 2r/3 - \delta \geq r/3\]
and we are done.

Now we proceed with the construction. Using Zorn's lemma, we find a maximal (under inclusion) tuple of open sets $(U_i)$  satisfying \ref{i:pf:stab:2} and \ref{i:pf:stab:3} above. We will show that it must also satisfy \ref{i:pf:stab:1}. If not, there exists $x_0 \in M_r$ and an  open neighborhood $W_0$ of $x_0$ such that $\partial(W_0,U_i) \geq \delta$ for all $i$. By lower semi-continuity of $\dtp$, there is an open neighborhood $W_1$ of $x_0$ such that $\dtp(W_1, g \cdot W_1) \geq r$. Suppose that there exists $j \leq n$ such that
\begin{itemize}
	\item $\dtp(g \cdot x_0, U_j) > r$;
	\item $\dtp(x_0, g \cdot U_j) > r$.
\end{itemize}
Since both conditions are open, there exists an open neighborhood $W_2$ of $x_0$ such that $\dtp(W_2, g \cdot U_j) \geq r$, and $\dtp(g \cdot W_2, U_j) \geq r$. This implies that if we set $W = W_0 \cap W_1 \cap W_2$, we can add $W$ to $U_j$ without violating \ref{i:pf:stab:2} or \ref{i:pf:stab:3}, thus contradicting the maximality of $(U_i)$. So our final task in order to obtain a contradiction is to find $j$ satisfying the two conditions above. First, note that
\begin{equation*}
|\set{i : \dtp(g \cdot x_0, U_i) \leq r}| \leq \ell.
\end{equation*}
Indeed, suppose to the contrary that there exist $y_{i_0} \in U_{i_0}, \ldots, y_{i_\ell} \in U_{i_\ell}$ with $\dtp(y_{i_s}, g \cdot x_0) < 2r$ for all $s \leq \ell$. Then the $y_{i_s}$ are $\ell+1$ points in a ball of radius $2r$ which are pairwise $\delta$ apart by \ref{i:pf:stab:2}, which contradicts the definition of $\ell$.

Similarly,
\begin{equation*}
|\set{i : \dtp(x_0, g \cdot U_i) \leq r}| \leq k,
\end{equation*}
because if there exist $y_{i_0} \in U_{i_0}, \ldots, y_{i_k} \in U_{i_k}$ with $\dtp(x_0, g \cdot y_{i_s}) < 2r$ for all $s$, then by the choice of $k$, there exist $s \neq t$ with $\partial(g \cdot y_{i_s}, g \cdot y_{i_t}) < \delta'$. Now the choice of $\delta'$ implies that $\dtp(y_{i_s}, y_{i_t}) < \delta$, contradicting \ref{i:pf:stab:2}. Now by the choice of $n$, there exists $j$ as desired.
\end{proof}

\begin{proof}[Proof of Theorem \ref{thm-stab-cont-mhp}]
  It is a general fact that the stabilizer map is upper semi-continuous, so we only have to prove lower semi-continuity. So for every open subset $O$ of $G$, we have to prove that
  \[ X_{G,O} \coloneqq \left\lbrace x \in X : G_x \cap O \neq \emptyset \right\rbrace \]
 is an open subset of $X$. Clearly it is enough to do this for every relatively compact open subset $O$.  
	
	Let $L$ be the subgroup of $G$ generated by $O$. The subgroup $L$ is open, so the $L$-flow $X$ is also Gleason complete (by Corollary~\ref{cor-Gleason complete-localcondition}). Moreover, $L$ is compactly generated, so, in particular, $\sigma$-compact. Since $X_{G,O} = X_{L,O}$, it follows that it is enough to prove the desired conclusion under the assumption that the group is $\sigma$-compact. From now on, we make this assumption.
	
	We fix $x_0 \in X_{G, O}$. Let $g \in O$ be such that $g \cdot x_0 = x_0$ and let $V \ni 1_G$ be open such that $Vg \sub O$. Now we find a neighborhood $U_0$ of $x_0$ that is contained in $X_{G,Vg} \sub X_{G, O}$. Since $G$ is $\sigma$-compact, by Proposition~\ref{prop-exist-contin-pseudonorm}, there are $r > 0$ and a continuous, proper, normal pseudo-norm $|| \cdot||$  on $G$ such that $B_{2r} \subseteq V$. If $\phi \colon X \to \R^n$ is a continuous function as given by Lemma~\ref{l:construct-function}, then we have
    \[ \begin{aligned}  U_0 :=  \left\lbrace x \in X : || \phi(g \cdot x) - \phi(x) ||_{\infty} < r/3 \right\rbrace & \subseteq \left\lbrace x \in X : \partial(g \cdot x,x) \leq r  \right\rbrace \\
        &  \subseteq \left\lbrace x \in X : x \in B_{2r}  g \cdot x  \right\rbrace \\
        &  \subseteq X_{G,Vg}.
    \end{aligned}\]
    So $U_0$ is an open neighborhood of $x_0$ that has the desired property. 
  \end{proof}

  \begin{remark}
    \label{rem:metrizable-extension}
    If $G$ is second countable and the flow $G \actson X$ is metrizable, then it is also possible to obtain a metrizable irreducible extension $X'$ of $X$ for which the stabilizer map is continuous. For this, it is enough to apply Lemma~\ref{l:construct-function} to some fixed proper norm $\nm{\cdot}$ on $G$, a countable, dense subset of $g \in G$ and all rational numbers $r$ to obtain a countable collection of $\phi \in C(\hat X_G)$ that will witness the continuity of $\Stab$. Then one can take $X'$ to be the spectrum of the (separable) closed, $G$-invariant sublattice of $C(\hat X_G)$ generated by $C(X)$ and this countable collection and the proof of the theorem goes through. (The reason is that the functions $\phi$ that we construct are $1$-Lipschitz with respect to $\dtp$, so that if $\phi$ is a witness for some $g \in G$, then it is also a witness for all $g'$ sufficiently close to $g$ with constants in \eqref{eq:prop-phi} perhaps slightly worse than $r$ and $r/3$.)
  \end{remark}


\section{Stabilizer flows}
\label{sec:stabilizer-flows}

Throughout this section, let $G$ be a locally compact group. The continuity of the stabilizer map allows us to associate to any Gleason complete flow $X$ a subflow of $\Sub(G)$, namely, the image of the stabilizer map. As every flow has a unique universal irreducible extension, this leads us to the following definition.

\begin{defn}
  \label{df:stabilizer-flow}
    Let $G$ be locally compact and let $G \actson X$ be a $G$-flow. The \df{stabilizer flow $\tS_G(X)$} of $X$ is the subflow of $\Sub(G)$ given by
  \begin{equation*}
    \tS_G(X) \coloneqq \Stab(\hat X_G) = \set{G_z : z \in \hat X_G}.
  \end{equation*}
\end{defn}

We have the following general facts about the stabilizer flow.
\begin{prop}
  \label{p:stab-flow}
  Let $G \actson X$ be a $G$-flow and let $\pi \colon \hat X_G \to X$ be the universal irreducible extension of $X$. Then the following hold:
  \begin{enumerate}
  \item \label{i:psf:nwdense} For any compact $K \sub G$, the set
    \begin{equation*}
      D_K \coloneqq \set{z \in \hat X_G : z \notin K \cdot z \And \pi(z) \in K \cdot \pi(z)}
    \end{equation*}
    is nowhere dense in $\hat X_G$.
  \item \label{i:psf:in-TG} For any dense subset $X' \sub X$, we have that $\tS_G(X) \sub \cl{\Stab(X')}$.
  \item \label{i:psf:pt-cont} If $x \in X$ is a point of continuity of $\Stab$, then $G_x \in \tS_G(X)$.
  \item \label{i:pt-cont-dens} If the set $X_0 \sub X$ of continuity points of $\Stab$ is dense in $X$, then $\tS_G(X) = \cl{\Stab(X_0)}$.
  \end{enumerate}
\end{prop}

\begin{proof}
  \ref{i:psf:nwdense} Let $U \sub \hat X_G$ be non-empty, open. We will find a non-empty, open subset of $U$ disjoint from $D_K$. Let $z_0 \in U \cap D_K$ (if there is no such $z_0$, we are done). The set $\set{(z, z') \in \hat X_G^2 : z \notin K \cdot z'}$ is open and $(z_0,z_0)$ belongs to it, so there exists a neighborhood $U'$ of $z_0$, $U' \sub U$, such that $K \cdot U' \cap U' = \emptyset$. By irreducibility of $\pi$, the set $\pi_*(U')$ is non-empty and for any $x \in \pi_*(U')$, we have that $x \notin K \cdot x$. Thus the open set $\pi^{-1}(\pi_*(U')) \sub U'$ is disjoint from $D_K$.

  \ref{i:psf:in-TG} Let $z_0 \in \hat X_G$ and let
  \begin{equation*}
    \cU = \set{H \in \Sub(G) : H \cap O_1 \neq \emptyset, \ldots, H \cap O_n \neq \emptyset, H \cap K = \emptyset},
  \end{equation*}
  where $O_1, \ldots, O_n \sub G$ are open and $K \sub G$ is compact, be a neighborhood of $G_{z_0}$ in $\Sub(G)$. Our goal is to find $x \in X'$ with $G_x \in \cU$. Let $U = \set{z \in \hat X_G : G_z \in \cU}$ and note that by continuity of the stabilizer map, $U$ is open. By \ref{i:psf:nwdense}, the open set $U \sminus \cl{D_K}$ is non-empty. We claim that any $x \in \pi_*(U \sminus \cl{D_K}) \cap X'$ works. Indeed, fix such an $x$ and let $z \in \hat X_G$ be such that $\pi(z) = x$. As $z \notin D_K$, we have that $G_x \cap K = \emptyset$ and as $G_z \leq G_x$, we also have that $G_x \cap O_i \neq \emptyset$ for all $i$, so $G_x \in \cU$.

  \ref{i:psf:pt-cont} Let $x$ be a point of continuity of $\Stab$ and let
  \begin{equation*}
    \cU \coloneqq \set{H \in \Sub(G) : H \cap O_1 \neq \emptyset, \ldots, H \cap O_n \neq \emptyset, H \cap K = \emptyset}
  \end{equation*}
  be a neighborhood of $G_{x}$, where each $O_i \sub G$ is open and $K \sub G$ is compact. Let $O_i' \sub G$ be open, relatively compact with $\cl{O_i'} \sub O_i$ such that $G_{x} \in \cU'$, where
  \begin{equation*}
    \cU' \coloneqq \set{H \in \Sub(G) : H \cap O_1' \neq \emptyset, \ldots, H \cap O_n' \neq \emptyset, H \cap K = \emptyset}.
  \end{equation*}
  By the continuity of $\Stab$ at $x$, there is an open $W \ni x$ with $\Stab(W) \sub \cU'$. By \ref{i:psf:nwdense}, the set $\bigcup_i D_{\cl{O'_i}}$ is nowhere dense, so there exists $z \in \pi^{-1}(W) \sminus \bigcup_i D_{\cl{O'_i}}$. Then $z \in \cl{O_i'} \cdot z \sub O_i \cdot z$ for every $i$ and $G_z \cap K = \emptyset$ (because $\pi(z) \in W$ and $G_z \leq G_{\pi(z)}$). Thus $G_z \in \cU$. As $\cU$ was arbitrary and $\tS_G(X)$ is closed, this implies that $G_x \in \tS_G(X)$.
  
  \ref{i:pt-cont-dens} follows from \ref{i:psf:in-TG} and \ref{i:psf:pt-cont}.
\end{proof}

The following is well-known and follows from \cite[Theorem~VII]{Kuratowski1928}. We include a short proof for completeness. 

\begin{lemma} \label{lem-set-pts-cont}
Let $X$ be a compact space, $Y$ a locally compact space, and let $\varphi \colon X \to 2^Y$ be upper semi-continuous. Let $(U_i)_{i \in I}$ be a basis for the topology on $Y$ such that each $U_i$ is relatively compact. For $i \in I$, we let \[X_i = \left\lbrace x \in X : \phi(x) \cap \cl{U_i} \neq \emptyset \right\rbrace .\] Then $\phi$ is continuous at each point of the set $\bigcap_i (X \setminus \partial X_i)$. 

In particular, if $Y$ is second countable, then the set of continuity points of $\phi$ is comeager.
\end{lemma}

\begin{proof}
Let $x \in \bigcap_i (X \setminus \partial X_i)$, and let $(x_a)$ be a net in $X$ converging to $x$ and such that $(\phi(x_a))$ converges to $F$. By upper semi-continuity, we know that $F \sub \phi(x)$, and we want to prove equality. Let $i$ such that $\phi(x) \cap \cl{U_i} \neq \emptyset$, i.e., $x \in X_i$. Since $x$ is in $X \setminus \partial X_i$ by assumption, $x$ must be in the interior of $X_i$. Since $(x_a)$ converges to $x$, eventually $x_a \in X_i$, that is, $\phi(x_a) \cap \cl{U_i} \neq \emptyset$. Since  $ \cl{U_i}$ is compact, this implies $F \cap \cl{U_i} \neq \emptyset$. So whenever $\phi(x)$ intersects $\cl{U_i}$, so does $F$. Since $(U_i)_{i \in I}$ is a basis for the topology on $Y$, this shows that $ \phi(x) \sub F$, as desired.

Note that $X_i$ is always closed by upper semi-continuity, so  $X \setminus \partial X_i$ is a dense open subset. In case $Y$ is second countable, $(U_i)_{i \in I}$ can be chosen to be countable, and hence the domain of continuity of  $\phi$ is comeager. 
\end{proof}

\begin{cor}
  \label{c:second-countable}
  Let $G$ be second countable and let $G \actson X$ be a $G$-flow. Then the set $X_0 \sub X$ of continuity points of $\Stab$ is dense $G_\delta$ in $X$ and we have
  \begin{equation*}
    \tS_G(X) = \cl{\Stab(X_0)}.
  \end{equation*}
\end{cor}

\begin{proof}
  The first claim follows from the upper semi-continuity of the stabilizer map and Lemma~\ref{lem-set-pts-cont}, and the second claim follows from \ref{i:pt-cont-dens} of Proposition~\ref{p:stab-flow}.
\end{proof}

\begin{remark}
When $G$ is not second countable, it is no longer true that there exists $x \in X$ such that $G_x \in \tS_G(X)$. Indeed, consider the group $G = \mathrm{SO}(3, \R)$, equipped with the discrete topology, acting on the $2$-dimensional sphere $X = \mathbf{S}^2$. Then $G_x \neq \set{1_G}$ for all $x \in X$. On the other hand, every non-identity element has only two fixed points in $X$, so the action is topologically free, which means that $\tS_G(X) = \set[\big]{\set{1_G}}$ (see Corollary   \ref{c:top-free}). Here the set of continuity points of $\Stab$ is empty.
\end{remark}

In the case where $X$ is minimal, stabilizer flows have already been considered in the literature under the name of stabilizer URSs. Recall that a \df{uniformly recurrent subgroup (URS)} of $G$ is a minimal subflow of $\Sub(G)$. Glasner and Weiss~\cite{Glasner2015} associated to every minimal $G$-flow its \df{stabilizer URS} as follows. Upper semi-continuity of the stabilizer map implies that $\cl{\Stab(X)}$ has a unique minimal subflow (see \cite[Lemma~1.1]{Auslander1977} or \cite[Proposition~1.2]{Glasner2015}). Then the \df{stabilizer URS of $X$} is simply defined to be this minimal subflow. Proposition~\ref{p:stab-flow} implies that for minimal flows, our definition and theirs coincide.

\begin{cor}
  \label{c:stab-flow-equal-URS}
  Let $X$ be a minimal $G$-flow. Then its stabilizer URS is equal to $\tS_G(X)$.
\end{cor}

\begin{proof}
  Proposition~\ref{p:stab-flow} \ref{i:psf:in-TG} tells us that $\tS_G(X) \sub \cl{\Stab(X)}$. As $X$ is minimal, $\hat X_G$ is also minimal and so is its factor $\tS_G(X)$. Now the conclusion follows from the fact that $\cl{\Stab(X)}$ has a unique minimal subflow.
\end{proof}
Corollary~\ref{c:second-countable} was also known for minimal $X$: see \cite{Glasner2015}*{Proposition~1.2}.

Recall that a flow $G \actson X$ is called \df{topologically free} if for every compact $K \sub G$ that does not contain $1_G$, the closed set $\set{x \in X : x \in K \cdot x}$ has empty interior. A point $x \in X$ is called \df{free} if the orbit map $G \to G \cdot x$, $g \mapsto g \cdot x$ is injective. A flow is called \df{free} if all points are free. It is clear that a flow for which the free points are dense is topologically free, and a simple Baire category argument shows that the converse is also true if $G$ is second countable. 

\begin{cor}
  \label{c:top-free}
  Let $G \actson X$ be a $G$-flow. Then the following are equivalent:
  \begin{enumerate}
  \item \label{i:ctf:X-top-free} $X$ is topologically free;
  \item \label{i:ctf:Gleason complete-free} $\hat X_G$ is free;
  \item \label{i:ctf:SGX-trivial} $\tS_G(X) = \set[\big]{\set{1_G}}$.
  \end{enumerate}
  In particular, topologically free Gleason complete flows are free.
\end{cor}

\begin{proof}
  The equivalence of \ref{i:ctf:Gleason complete-free} and \ref{i:ctf:SGX-trivial} follows from the definition of $\tS_G(X)$.

  \ref{i:ctf:X-top-free} $\Rightarrow$ \ref{i:ctf:Gleason complete-free} 
  Since $\hat X_G \to X$ is irreducible, the assumption that $X$ is topologically free implies that $\hat X_G$ is also topologically free. Let $g \in G$, $g \neq 1_G$. Let $V \sub G$ be an open, relatively compact subset with $g \in V$ and $1_G \notin \cl{V}$. Then the set $\set{z \in \hat X_G : z \in V \cdot z}$ is open by Theorem~\ref{thm-stab-cont-mhp}, and has empty interior by topological freeness, so it must be empty. So we conclude that $g \cdot z \neq z$ for all $z$.

  \ref{i:ctf:Gleason complete-free} $\Rightarrow$ \ref{i:ctf:X-top-free} Suppose, towards a contradiction, that there is a compact $K \sub G$ with $1_G \notin K$ such that the set $\set{x \in X : x \in K \cdot x}$ has non-empty interior $W$. By Proposition~\ref{p:stab-flow} \ref{i:psf:nwdense}, the set $\pi^{-1}(W) \sminus D_K$ is non-empty and for any $z$ in this set, we have that $z \in K \cdot z$, contradicting the freeness of $\hat X_G$.
\end{proof}

From this, it is not hard to deduce a well-known theorem of Veech.
\begin{cor}[Veech]
  \label{c:Veech}
  Every locally compact group admits a free flow.
\end{cor}

\begin{proof}
  Let $G$ be a locally compact group and let $\Sam(G)$ denote its \df{Samuel compactification}, i.e., the spectrum of the Riesz space of right uniformly continuous bounded functions on $G$. Then $G \actson \Sam(G)$ is a $G$-flow and $G$ embeds densely in $\Sam(G)$ as point evaluations. Also, the flow $\Sam(G)$ is Gleason complete by \cite{Zucker2021}*{3.2.1} (alternatively, it is not difficult to verify condition \ref{item-cont-dist} of Theorem~\ref{th:char-distance}). As the left translation action $G \actson G$ is free, Corollary~\ref{c:top-free} tells us that the flow $\Sam(G)$ is also free.
\end{proof}

\bibliography{stab-irreducible-ext}
\end{document}